\documentclass{article}

\usepackage{graphicx}
\usepackage{natbib}
\bibpunct{(}{)}{;}{a}{}{,}
\usepackage{lipsum,setspace,multirow,array}
\usepackage{mleftright}
\usepackage{amsmath,amssymb,amstext,amsthm}

\newcommand{\pren}[1]{\mleft(#1\mright)}

\newcommand{\tox}[1]{\overset{#1}{\to}}

\newcommand{\Bc}{\mathcal{B}}

\newcommand{\Cb}{\mathbb{C}}
\newcommand{\Eb}{\mathbb{E}}
\newcommand{\Nb}{\mathbb{N}}

\newcommand{\Rb}{\mathbb{R}}
\newcommand{\Rc}{\mathcal{R}}
\newcommand{\LT}{\ensuremath{\mathcal{L}}}

\newcommand{\leqorder}[1]{\ensuremath{\preceq_{{\rm #1}}}}
\newcommand{\geqorder}[1]{\ensuremath{\preceq_{{\rm #1}}}}

\renewcommand{\d}[1]{\ensuremath{\mathtt{d}{#1}}}
\newcommand{\SIR}{\ensuremath{\mathrm{SIR}}}

\theoremstyle{plain}
\newtheorem{theorem}{\textbf{Theorem}}
\newtheorem{corollary}{\textbf{Corollary}}
\newtheorem{lemma}{\textbf{Lemma}}
\newtheorem{property}{\textit{Property}}

\theoremstyle{definition}
\newtheorem{definition}{\textbf{Definition}}
\newtheorem*{definition*}{\textbf{Definition}}

\theoremstyle{remark}
\newtheorem*{remark*}{\textbf{Remark}}

\newcommand{\lemref}[1]{Lemma~\ref{#1}}
\newcommand{\thmref}[1]{Theorem~\ref{#1}}

\newcommand{\corref}[1]{Corollary~\ref{#1}}

\title{Contribution of the Extreme Term in the Sum of Samples with Regularly Varying Tail}

\author{Van Minh Nguyen \thanks{Huawei Technologies, France. Email: vanminh.nguyen@huawei.com}}

\begin{document}

\maketitle

\begin{abstract}
	For a sequence of random variables $(X_1, X_2, \ldots, X_n)$, $n \geq 1$, that are independent and identically distributed with a regularly varying tail with index $-\alpha$, $\alpha \geq 0$, we show that the contribution of the maximum term $M_n \triangleq \max(X_1,\ldots,X_n)$ in the sum $S_n \triangleq X_1 + \cdots +X_n$, as $n \to \infty$, decreases monotonically with $\alpha$ in stochastic ordering sense.\\
	
	
\end{abstract}

\section{Introduction}\label{sec:Intro}
	
	Let $(X_1, X_2, \ldots)$ be a sequence of random variables that are independent and identically distributed (i.i.d.) with distribution $F$. For $n \geq 1$, the extreme term and the sum are defined as follows: 
	\begin{equation}
		M_n \triangleq \max_{i=1}^n X_i, \quad S_n \triangleq \sum_{n=1}^n X_i.
	\end{equation}
	The influence of the extreme term in the sum has various implications in both theory and applications. In particular, it has been used to characterize the nature of possible convergence of the sums of i.i.d. random variables \citep{Darling1952}. On the other hand, besides well-known applications to risk management, insurance and finance \citep{Embrechts1997}, it has been recently applied to wireless communications for characterizing a fundamental parameter which is the signal-to-interference ratio \citep{Nguyen2010,Nguyen2017}.
	
	Denote by $\bar{F}(x) \triangleq 1 - F(x)$ the tail distribution of $F$, a primary result of this question is due to \citep{Darling1952} in which it is shown that
	\begin{property}[\citep{Darling1952}]
		Suppose that $X_i \geq 0$. Then $\Eb\pren{S_n/M_n} \to 1$ as $n \to \infty$ if for every $t > 0$ we have
		\begin{equation}\label{eq:DarlingCond}
			\lim_{x \to \infty} \frac{\bar{F}(tx)}{\bar{F}(x)} = 1.
		\end{equation}
	\end{property} 
	
	In Darling's result, condition \eqref{eq:DarlingCond} is of particular interest as it is a specific case of a more general class called \emph{regularly varying tails} which is defined in the following \citep{Bingham1989}:
	
	\begin{definition} 
		A positive, Lebesgue measurable function $h$ on $(0,\infty)$ is called \emph{regularly varying} with index $\alpha \in \Rb$ at $\infty$ if $\displaystyle \lim_{x \to \infty} \frac{h(tx)}{h(x)} = t^{\alpha}$ for every constant $0 < t < \infty$.
	\end{definition}
	In the sequel, $\Rc_{\alpha}$ denotes the class of regularly varying functions with index $\alpha$, and in particular $\Rc_{0}$ is referred to as the class of \emph{slowly varying} functions. In addition, $\tox{d}$, $\tox{p}$, and $\tox{a.s.}$ stands for the convergence in \emph{distribution}, convergence in \emph{probability}, and \emph{almost sure} convergence, respectively.
		
	Since the work of \citep{Darling1952}, there has been subsequent extensions which in particular investigated other cases of regularly varying tails. Among those, \citep{Arov1960} derived the characteristic function and limits of the jointed sum and extreme term for a regularly varying tail. \citep{Teugels1981} derived the limiting characteristic function of the ratio of the sum to order statistics, and moreover investigated norming sequences for its convergence to a constant or a normal law.  \citep{Chow1978,Anderson1991,Anderson1995} investigated the asymptotic independence of normed extreme and normed sum. Unlike the slowly varying case, \citep{OBrien1980} showed that $M_n/S_n \tox{a.s.} 0 \Leftrightarrow \Eb{X_1} < \infty$. For $\bar{F}$ regularly varying with index $-\alpha$, $\alpha > 0$, \citep{Bingham1981} showed that the extreme term only contributes a proportion to the sum:
	
	\begin{property}[\citep{Bingham1981}]
		The following are equivalent:
		\begin{enumerate}
			\item $\bar{F} \in \Rc_{-\alpha}$ for $0 < \alpha < 1$;
			\item $M_n/S_n \tox{d} R$ where $R$ has a non-degenerate distribution;
			\item $\Eb\pren{S_n/M_n} \to (1-\alpha)^{-1}$.
		\end{enumerate}
	\end{property}
	
	\begin{property}[\citep{Bingham1981}]
		Let $\mu = \Eb X_1$. The following are equivalent:
		\begin{enumerate}
			\item $\bar{F} \in \Rc_{-\alpha}$ for $1 < \alpha < 2$;
			\item $(S_n - (n-1)\mu)/M_n \tox{d} D$ where $D$ has a non-degenerate distribution;
			\item $\Eb\pren{(S_n - (n-1)\mu)/M_n} \to c$ where $c$ is a constant.
		\end{enumerate}
	\end{property}
	
	\citep{Maller1984} extended Darling's convergence in mean to convergence in probability of $S_n/M_n$:
	
	\begin{property}[\citep{Maller1984}]
		\begin{equation*}
			S_n/M_n \tox{p} 1 \Leftrightarrow \bar{F} \in \Rc_0.
		\end{equation*}
	\end{property}
	
	The ratio of the extreme to the sum has been further studied with the following result:
	\begin{property}[\citep{Downey1994}]
		If either one of the following conditions:
		\begin{enumerate}
			\item $\bar{F} \in \Rc_{-\alpha}$ for $\alpha > 1$,
			\item $F$ has finite second moment,
		\end{enumerate}
		holds, then
		\begin{equation*}
			\Eb\pren{M_n/S_n} = \frac{\Eb M_n}{\Eb S_n}\pren{1 + o(1)}, \quad \textrm{ as } n \to \infty.
		\end{equation*}
	\end{property}	 
	
	To this end, the contribution of the extreme term in the sum has been investigated and classified in the following cases:
	\begin{itemize}
		\item $\bar{F} \in \Rc_0$; 
		\item $\bar{F} \in \Rc_{-\alpha}$, $0 < \alpha < 1$; 
		\item $\bar{F} \in \Rc_{-\alpha}$, $1 < \alpha < 2$; 
		\item $\bar{F} \in \Rc_{-\alpha}$, $\alpha > 2$. 
	\end{itemize}
	Nevertheless, how the influence of the extreme term in the sum gradually varies with the regular variation index has not been quantified and remains an open question. Precisely, \emph{consider two cases with $\bar{F} \in \Rc_{-\alpha_1}$ and $\bar{F} \in \Rc_{-\alpha_2}$ in which $0 \leq \alpha_1 < \alpha_2$, which case results in larger $M_n/S_n$?} This question is particularly important for analysis and design of wireless communication networks. In this context, random variables $(X_1, X_2, \ldots)$ are used to model the signal that a user receives from base stations. It has been proven that the tail distribution of $X_i$ is regularly varying \citep{Nguyen2017} either due to the effect of distance-dependent propagation loss or due to fading \citep{Tse2005} that is regular varying \citep{Rajan2017} as a consequence of advances in communication and signal processing techniques such as massive multiple-input-multiple-output transmission, coordinated multipoint and millimeter wave systems. Meanwhile, $M_n$ expresses the useful signal power and $(S_n-M_n)$ is the total interference due to the other transmitters \citep{Nguyen2011}. $M_n/(S_n-M_n)$ is hence the signal-to-interference ratio (SIR), and its limit as $n \to \infty$ happens for a dense or ultra-dense network \citep{Nguyen2016,Nguyen2017a}. Capacity of a communication channel is expressed in term of the well-known Shannon's capacity limit of $\log(1 + \mathrm{SIR})$ \citep{Shannon1948,Cover2006} considering that thermal noise is negligible in comparison to the interference. Therefore, $M_n/S_n$ (or $S_n/M_n$) is a fundamental parameter of wireless network engineering. From the perspective of capacity, a primary purpose is to design the network such that $X_i$ possesses properties that make $M_n/S_n$ as large as possible. In particular, how $M_n/S_n$ varies according to the tail of $X_i$ turns out to be a critical question.
	
	In this paper, we establish a stochastic ordering for $S_n/M_n$ and show that between two cases with $\bar{F} \in \Rc_{-\alpha_1}$ and $\bar{F} \in \Rc_{-\alpha_2}$ in which $0 \leq \alpha_1 < \alpha_2$, the contribution of $M_n$ in $S_n$ as $n \to \infty$ is larger in the former than in the latter case in stochastic ordering sense.
	
\section{Main Result}\label{sec:Main}

	In the following, for $n \geq 1$ we define
	\begin{equation}\label{eq:DefRn}
		R_{n} \triangleq S_n/M_n.
	\end{equation}
	We also restrict our consideration to non-negative random variables, i.e., $X_i \geq 0$, and consider $\bar{F} \in \Rc_{-\alpha}$ with $\alpha \geq 0$. In the context where $\alpha$ is analyzed, a variable $v$ is written as $v_{\alpha}$, e.g., write $S_{\alpha,n}$, $M_{\alpha,n}$, and $R_{\alpha,n}$ for $S_{n}$, $M_{n}$, and $R_{n}$, respectively.
	
	\begin{lemma}\label{lem:Laplace}
		For $s \in \Cb$ with $\Re(s) \geq 0$, define $\LT_{R_n}(s) \triangleq \Eb\pren{e^{-s R_n}}$. If $\bar{F} \in \Rc_{-\alpha}$ with $\alpha \geq 0$ then:
		\begin{equation}
			\LT_{R_n}(s) = \frac{e^{-s}}{1 + \phi_{\alpha}(s)}, \quad n \to \infty,
		\end{equation}
		where
		\begin{equation}\label{eq:phi}
			\phi_{\alpha}(s) \triangleq \alpha\int_0^1 (1 - e^{-st})\frac{\d{t}}{t^{1+\alpha}}.
		\end{equation}
	\end{lemma}
	
	\begin{proof}
		Let $G(x_1,\cdots,x_n)$ be the joint distribution of $(X_1,\cdots,X_n)$ given that $M_n = X_1$, it is given as follows:
		\begin{equation}\label{eq:G}
			G(\d{x_1},\cdots,\d{x_n}) = \begin{cases}
			F(\d{x_1})\cdots F(\d{x_n}) & \text{if } x_1 = \max_{i=1}^n x_i \\
			0 & \text{otherwise}
		\end{cases}.
		\end{equation}
		Since $(X_1,\cdots,X_n)$ are i.i.d., $M_n = X_1$ with probability $1/n$. Thus, the (unconditional) joint distribution of $(X_1,\cdots,X_n)$ is $nG(x_1,\cdots,x_n)$. Hence,
		\begin{equation*}
			\LT_{R_n}(s) = \Eb(e^{-s R_n}) = n\int\cdots\int e^{-s(x_1+x_2 + \cdots + x_n)/x_1} G(\d{x_1},\cdots,\d{x_n}),
		\end{equation*}
		and using $G$ from \eqref{eq:G}, we obtain
		\begin{align}
			\LT_{R_n}(s) & = n \int_0^{\infty} \pren{\int_0^{x} \cdots \int_0^x e^{-s} \prod_{i=2}^n e^{-s x_i/x} F(\d{x_i})} F(\d{x})\nonumber\\
			& = n e^{-s} \int_0^{\infty} \left( \int_0^{1} e^{-st} F(x\d{t})\right)^{n-1} F(\d{x}) \nonumber\\
			& = n e^{-s} \int_0^{\infty} \left(\varphi(x)\right)^{n-1} F(\d{x}), \label{eq:LaplaceZn}
		\end{align}
		where 
		\begin{equation}
			\varphi(x) \triangleq \int_0^{1} e^{-st}F(x\d{t}).
		\end{equation} 
		Given that $\Re(s) \geq 0$, we can see that
		\begin{equation*}
			|\varphi(x)| \leq \int_0^1 |e^{-st} F(x\d{t})| \leq \int_0^1 F(x\d{t}) = F(x) < 1, \quad \text{for } x < \infty. 
		\end{equation*}
		Hence, 
		\begin{equation*}
			\int_0^{T} \left(\varphi(x)\right)^{n-1} F(\d{x}) \to 0, \quad \textrm{as } n \to \infty  \textrm{ for } T < \infty.
		\end{equation*}
		As a result, we only need to consider the contribution of large $x$ in \eqref{eq:LaplaceZn} for $\LT_{R_n}(s)$. An integration by parts with $e^{-st}$ and $F(x\d{t})$ yields:
		\begin{align}
			\varphi(x) & = 1 - e^{-s}\bar{F}(x) - \int_0^1 s e^{-st} \bar{F}(xt) \d{t} \nonumber\\
			& = 1 - \bar{F}(x) + \int_0^1 s e^{-st} \left(\bar{F}(x) - \bar{F}(tx)\right)\d{t}. \label{eq:varphi}
		\end{align}
		For $\bar{F} \in \Rc_{-\alpha}$ with $\alpha \geq 0$, we can write 
		\begin{equation*}
			\bar{F}(tx) \sim t^{-\alpha}\bar{F}(x), \quad \textrm{ as } x \to \infty, \quad 0 < t < \infty.
		\end{equation*}
		Thus
		\begin{equation*}
			\int_0^1 s e^{-st} \left(\bar{F}(x) - \bar{F}(tx)\right)\d{t} 
			\sim \bar{F}(x)\int_0^1 s e^{-st} (1 - t^{-\alpha})\d{t}, \quad \textrm{as } x \to \infty.
		\end{equation*}
		Using an integration by parts with $(1 - t^{-\alpha})$ and $\d{(1-e^{-st})}$ we obtain
		\begin{equation}\label{eq:phi_2ndterm}
			\bar{F}(x)\int_0^1 s e^{-st} (1 - t^{-\alpha})\d{t} = -\bar{F}(x) \int_0^1 \alpha (1 - e^{-st}) \frac{\d{t}}{t^{1+\alpha}}.
		\end{equation}
		Put 
		\begin{equation*}
			\phi_{\alpha}(s) \triangleq \int_0^1 \alpha (1 - e^{-st}) \frac{\d{t}}{t^{1+\alpha}},
		\end{equation*}
		and substitute it back in \eqref{eq:phi_2ndterm} and \eqref{eq:varphi}, we obtain
		\begin{equation*}
			\varphi(x) = 1 - (1 + \phi_{\alpha}(s))\bar{F}(x), \quad \text{for } x \to \infty.
		\end{equation*}
		Substitute $\varphi(x)$ back in the expression of $\LT_{R_n}(s)$ in \eqref{eq:LaplaceZn}, we obtain
		\begin{equation}
			\LT_{R_n}(s) \sim n e^{-s}\int_0^{\infty} \left( 1 - (1 + \phi_{\alpha}(s))\bar{F}(x) \right)^{n-1} F(\d{x}), \quad \text{ as } n \to \infty.
		\end{equation}
		Here, we resort to a change of variable with $v = n \bar{F}(x)$ and obtain:
		\begin{align*}
		\LT_{R_n}(s) 
		& \sim e^{-s} \int_0^n \left(1 - \frac{v}{n}(1+\phi_{\alpha}(s))\right)^{n-1}\d{v} \\
		& \overset{(a)}{\to} e^{-s}\int_0^\infty e^{-v(1+\phi_{\alpha}(s))}\d{v} \\
		& = \frac{e^{-s}}{1 + \phi_{\alpha}(s)}, \quad \text{ as } n \to \infty 
		\end{align*}
		where $(a)$ is due to the formula $(1 + \frac{x}{n})^n \to e^x$ as $n \to \infty$.
	\end{proof}
	
	To present the main result, we use the following notation. For two random variables $U$ and $V$, $U$ is said to be smaller than $V$ in Laplace transform order, denoted by $U \leqorder{Lt} V$, if and only if $\LT_U(s)=\Eb(e^{-sU}) \geq \Eb(e^{-sV}) = \LT_V(s)$ for all positive real number $s$. 
	
	\begin{theorem}\label{thm:OrderRn}
		Let $R_{\alpha_1,n}$ and $R_{\alpha_2,n}$ be as defined in \eqref{eq:DefRn} for $\bar{F} \in \Rc_{-\alpha_1}$ with $\alpha_1 \geq 0$ and for $\bar{F} \in \Rc_{-\alpha_2}$ with $\alpha_2 \geq 0$, respectively. Then
		\begin{equation}
			\alpha_1 \leq \alpha_2 \Rightarrow R_{\alpha_1,n} \geqorder{Lt} R_{\alpha_2,n}, \quad n \to \infty.
		\end{equation}
	\end{theorem}
	\begin{proof}
		The proof is direct from \lemref{lem:Laplace}. By noting that $\alpha/t^{1+\alpha}$ is increasing with respect to (w.r.t.) $\alpha \geq 0$ for $t \in [0,1]$, we have $\phi_{\alpha}(s)$ in \eqref{eq:phi} is increasing  w.r.t. $\alpha \geq 0$. Thus, $\LT_{R_{\alpha,n}}(s)$ as $n \to \infty$ and $\Re(s) > 0$ is decreasing w.r.t. $\alpha \geq 0$.
		
		It follows that, for $0 \leq \alpha_1 \leq \alpha_2$, we have $\LT_{R_{\alpha_1,n}}(s) \geq \LT_{R_{\alpha_2,n}}(s)$ as $n \to \infty$ for all $s$ with $\Re(s) > 0$, thus $R_{\alpha_1,n} \leqorder{Lt} R_{\alpha_2,n}$ as $n \to \infty$.
	\end{proof}
	
	\thmref{thm:OrderRn} dictates that the more slowly $\bar{F}$ decays at $\infty$, i.e., smaller $\alpha$, the smaller is the ratio of the sum to the extreme, thus the bigger is the contribution of the extreme term in the sum. This contribution of the extreme term to the sum increases to the ceiling limit 1 when $\alpha$ gets close to 0 as we have known from \citep{Darling1952,Maller1984}.
		
	Since we have established the Laplace transform ordering for $R_{n}$, an immediate application is related to completely monotonic and Bernstein functions. Let us recall:
	\begin{itemize}
		\item \emph{Completely monotonic functions}: A function $g: (0,\infty) \rightarrow \Rb_+$ is said to be completely monotonic if it possesses derivatives of all orders $k \in \Nb \cup \{0\}$ which satisfy $(-1)^{k} g^{(k)}(x) \geq 0$, $\forall x \geq 0$, where the derivative of order $k=0$ is defined as $g(x)$ itself. We denote by $\mathcal{CM}$ the class of completely monotonic functions.
		\item \emph{Bernstein functions}: A function $h: (0,\infty) \rightarrow \Rb_+$ with $\d{h(x)}/\d{x}$ being completely monotonic is called a Bernstein function. We denote by $\mathcal{B}$ the class of Bernstein functions.
	\end{itemize}
	Note that a completely monotonic function is positive, decreasing and convex, whereas a Bernstein function is positive, increasing and concave.	
	
	It is well known that for all completely monotonic functions $g$, we have that $U \leqorder{Lt} V \Leftrightarrow \Eb\pren{g(U)} \geq \Eb\pren{g(V)}$, whereas for all Bernstein functions $h$, $U \leqorder{Lt} V \Leftrightarrow \Eb\pren{h(U)} \leq \Eb\pren{h(V)}$. Hence, we can have a direct corollary of \thmref{thm:OrderRn} as follows.
	
	\begin{corollary}\label{cor:MeanOrdering}
		With the same notation and assumption of \thmref{thm:OrderRn}, if $0 \leq \alpha_1 \leq \alpha_2$, then:
		\begin{align*}
			\forall g \in \mathcal{CM}: & \quad \Eb\pren{g(R_{\alpha_1,n})} \geq \Eb\pren{g(R_{\alpha_2,n})}, \quad \textrm{as } n \to \infty, \\
			\forall h \in \Bc: & \quad \Eb\pren{h(R_{\alpha_1,n})} \leq \Eb\pren{h(R_{\alpha_2,n})}, \quad \textrm{as } n \to \infty.
		\end{align*}
	\end{corollary}
	
	Note that $h(x) = 1$, $\forall x >0$, is a Bernstein function, whereas $g(x) = 1/x$, $\forall x > 0$, is a completely monotonic function. For two cases with $\bar{F} \in \Rc_{-\alpha_1}$ and $\bar{F} \in \Rc_{-\alpha_2}$ with $0 \leq \alpha_1 \leq \alpha_2$, \corref{cor:MeanOrdering} gives:
	\begin{align*}
		\Eb\pren{\frac{S_{\alpha_1,n}}{M_{\alpha_1,n}}} & \leq \Eb\pren{\frac{S_{\alpha_2,n}}{M_{\alpha_2,n}}}, \quad n \to \infty,\\
		\Eb\pren{\frac{M_{\alpha_1,n}}{S_{\alpha_1,n}}} & \geq \Eb\pren{\frac{M_{\alpha_2,n}}{S_{\alpha_2,n}}}, \quad n \to \infty.
	\end{align*}

	\paragraph{Application Example}
	We now can show an application of the results developed above to the context of wireless communication networks. The signal-to-interference ratio $\SIR$ as described in the Introduction can be expressed in terms of $R_n$ as follows:
	\begin{equation}
		\frac{1}{\SIR} = \frac{S_n - M_n}{M_n} = R_n - 1 := Z_n.
	\end{equation}
	Assume that $\bar{F} \in \Rc_{-\alpha}$, $\alpha \geq 0$, the Laplace transform of $Z_n$ can be directly obtained from that of $R_n$ as given by \lemref{lem:Laplace} as follows:
	\begin{equation}
		\LT_{Z_{\alpha,n}}(s) = e^{s}\LT_{R_{\alpha,n}}(s) = \pren{1 + \phi_{\alpha}(s)}^{-1}, \quad n \to \infty,
	\end{equation}
	for all $s \in \Cb$, $\Re{s} > 0$. It is easy to see that $\LT_{Z_{\alpha,n}}(s)$ is also decreasing with respect to $\alpha \geq 0$ (see the proof of \thmref{thm:OrderRn}). 
	
	Now, consider two cases for the distribution of the signal $X_i$ which are $\bar{F} \in \Rc_{-\alpha_1}$ and $\bar{F} \in \Rc_{-\alpha_2}$ with $0 \leq \alpha_1 \leq \alpha_2$, we firstly have:
	\begin{equation}
		Z_{\alpha_1,n} \leqorder{Lt} Z_{\alpha_2,n}, \quad n \to \infty.
	\end{equation}	
	Then, by noting that functions $1/x$ and $\log(1 + 1/x)$ with $x > 0$ both are completely monotonic, we immediately have for $n \to \infty$:
	\begin{align}
		\Eb(\SIR_{\alpha_1}) & \geq \Eb(\SIR_{\alpha_2}), \\
		\Eb(\log(1+\SIR_{\alpha_1})) & \geq \Eb(\log(1 + \SIR_{\alpha_2})).		
	\end{align}
	This says that it is beneficial to the communication quality and capacity to design the network such that the signal received from a transmitting base station $X_i$ admits a regularly varying tail with as small variation index as possible (i.e., as close to a slowly varying tail as possible).
	

\bibliographystyle{authordate1}
\bibliography{extremes}

\end{document}